\documentclass[12pt]{amsart}
\usepackage[square,longnamesfirst]{natbib}

\newtheorem{theo}{Theorem}
\newtheorem{cor}[theo]{Corollary}
\theoremstyle{remark}

\def \p {\mathfrak p}
\def \Ql {\mathbb Q_\ell}
\def \Gal {\operatorname{Gal}}
\def \Frob {\operatorname{Frob}}
\def \tr {\operatorname{tr}}

\def \Zar {\operatorname{Zar}}

\begin{document}

\title{Ordinary Primes for Abelian Surfaces}
\author{William F. Sawin}
\thanks{The author is supported by the National Science Foundation Graduate Research Fellowship under Grant No. DGE-1148900}

\maketitle

\begin{abstract} We compute the density of the set of ordinary primes of an abelian surface over a number field in terms of the $\ell$-adic monodromy group. Using the classification of $\ell$-adic monodromy groups of abelian surfaces by Fit\'{e}, Kedlaya, Rotger, and Sutherland, we show the density is $1$, $1/2$, or $1/4$. 

\end{abstract}

\vspace{10pt}

We study the density of the set of primes at which an abelian surface has ordinary reduction. This density is known to be positive \citep[pp.370-372]{ogus82}, and to equal $1$ when the endomorphism ring is $\mathbb Z$ \citep[Theorem 7.1]{pink98}. In this paper we completely resolve the density question for abelian surfaces, by refining the the  $\ell$-adic method of Serre, as applied by Katz to abelian surfaces and explained in \citep[pp.370-372]{ogus82}. The density is always $1$, $1/2$, or $1/4$, and we describe when each occurs.

I thank Nick Katz for helpful conversations and J.P. Serre for helpful emails.

\vspace{10pt}

Let $A$ be an abelian surface over a number field $K$. Fix a prime number $\ell$. Let $G \subseteq GSP_4 $ be the $\ell$-adic monodromy group of $A$ - by definition, the Zariski closure of the image of $\Gal(\overline{\mathbb Q}\slash K)$ inside $GSP_4$ under the map defined by the action of $\Gal(\overline{\mathbb Q}\slash K)$ on $H^1(A, \Ql)$. (Here $GSP_4$ is viewed as an algebraic group over $\Ql$.) Let $V$ be the standard representation of $GSP_4$ and let $\chi$ be the similitude character. Then we can compute the density of the ordinary primes of $A$ in terms of the action of $G$ on $ \wedge^2 V \otimes \chi^{-1}$:

\begin{theo}\label{main}  The density of the set of non-ordinary primes of $A$ is equal to the number of connected components\footnote{It does not matter whether we consider connected components of the scheme $G$ or of its geometric form $G_{\overline{\mathbb Q}_\ell}$. Because $G$ is defined to be the Zariski closure of a subset of $GSP_4(\Ql)$, each connected component of $G_{\overline{\mathbb Q}_\ell}$ contains a point of $GSP_4(\Ql)$ and hence is defined over $\Ql$.} of $G$  on which the trace of the representation $\wedge^2 V \otimes \chi^{-1}$ is a constant function, divided by the number of connected components of $G$. \end{theo}

\begin{proof} The primes $\p$ of $K$ that are split over primes $p\neq \ell$ of $\mathbb Q$, such that $A$ has good reduction at $\p$, have density one. Let $\p$ be such a prime.

The characteristic polynomial of $\Frob_\p$ acting on $H^1(A,\Ql)$ has the form
\[ x^4 - a_1 x^3 + a_2 x^2  - p a_1 x + p^2 \]

\noindent for integers $a_1,a_2$. Recall that $A$ has ordinary reduction if and only if $a_2$ is not a multiple of $p$ \citep[pp.238 (IV)]{deligne69}.

By the Weil bound, $a_2$  is in the interval $[-6p,6p]$. (In fact one can show it is in the interval $[-2p,6p]$.) So $a_2$ is a multiple of $p$ if and only if it is equal to $np$ for some integer $n\in [-6,6]$. We may compute both $a_2$ and $p$ in terms of $\Frob_\p \in GSP_4(\Ql)$: the trace of $\Frob_\p$ acting on $\wedge^2 V$ is $a_2$, and, because the symplectic form comes from the Weil pairing, the action of $\Frob_p$ on the similitude character $\chi$ is multiplication by $p$. Hence:
\[ \frac{a_2}{p}  = \tr (\Frob_\p, \wedge^2 V \otimes \chi^{-1} )\]

Because $\wedge^2 V \otimes \chi^{-1}$ is an algebraic representation, the set where the trace has a given value is a Zariski closed subset, so is a closed subset in the $\ell$-adic topology. Let $Z$ be the finite union over all integers $n \in [-6,6]$ of the closed subset of $G$ where the trace is $n$. Then $Z$ is a conjugacy-invariant closed subset of $G$. Let $\Gamma$ be the image of $\Gal(\overline{\mathbb Q}\slash K)$ in $GSP_4(\Ql)$. $\Gamma$ is a closed subgroup of $GSP_4(\mathbb Q_\ell)$, hence an $\ell$-adic analytic group \citep[\S8.2 Theorem 2]{bourbakilie}. $\Gamma$ is also compact, so it has a Haar measure of total mass one. $Z\cap \Gamma $ is an analytic subset of $\Gamma$, so its boundary has measure $0$ \citep[Proposition 5.9]{serre11}. Hence, by Chebotarev's density theorem, the density of primes lying in $Z$ is equal to the Haar measure of $Z \cap \Gamma$  \citep[Corollary 6.10]{serre11}. Because $\Gamma$ is Zariski dense in $G$, the Haar measure of $Z \cap \Gamma$ is equal to the number of connected components of $G$ contained in $Z$ divided by the number of connected components of $G$ \citep[Proposition 5.12 and 5.2.1.2]{serre11}.

Thus the density of the set of non-ordinary primes is equal to the number of connected components where the trace is constant and equal to $n$ for some $n  \in [-6,6]$ divided by the number of connected components. So it is sufficient to show that on every connected component where the trace of the representation is constant, the trace is equal to one of those $13$ values. If there is a connected component where the trace is constant and equal to $c$, then by Chebotarev again, for infinitely many split primes the trace must equal $c$. At these primes $a_2$ is equal to $cp$. The coefficients of the characteristic polynomial, in particular $a_2$, are always integers, so $c$ must be a rational number whose denominator divides $p$. Because this occurs for infinitely many, hence at least two, different primes $p$, $c$ is an integer. Then because $a_2 \in [-6p,6p]$, $c\in[-6,6]$.  Therefore, the density of the set of ordinary primes is equal to the proportion of connected components where the trace is nonconstant. \end{proof}

An immediate corollary is:

\begin{cor}\label{connected} If $G$ is connected, then the set of ordinary primes has density one. \end{cor}

\begin{proof} By Theorem \ref{main}, it is sufficient to show that the trace of $\wedge^2 V \otimes \chi^{-1}$ is not constant on $G$. Because the identity matrix is in $G$, if the trace is constant it is equal to the trace of the identity matrix, namely $6$.  Hence for every split prime $\p$, we would have $a_2=6p$. But in terms of the four eigenvalues $\alpha_1,\alpha_2,\alpha_3,\alpha_4$ of Frobenius on $H^1(A,\Ql)$, $a_2$ is the sum of the six products $\alpha_i \alpha_j$ for $i<j$. As each $|\alpha_i|=\sqrt{p}$, the only way $a_2$ can be $6p$ is if all the eigenvalues are $\sqrt{p}$ or all are $-\sqrt{p}$. But this is impossible, because then $a_1$ would be $4\sqrt{p}$ or $-4\sqrt{p}$, which is not an integer. \end{proof}

We can make Theorem \ref{main} more explicit using the classification of \citep{fkrs}, which lists all possibilities for the ``Sato-Tate group" of an abelian surface. The Sato-Tate group determines the base change of the $\ell$-adic monodromy group to $\overline{\mathbb Q}_\ell$. By applying Theorem \ref{main} to each of these monodromy groups, we obtain: 

\begin{theo}The density of the set of ordinary primes of $A$ is $1$ unless $A$ either is a CM abelian surface, or is isogenous to the product of a CM elliptic curve and a non-CM elliptic curve, or is isogenous to the product of CM elliptic curves. In these cases, the density of the set of ordinary primes is:

\begin{itemize}

\item If $A$ is a CM abelian surface and $F$ is the smallest field that all the endomorphisms of $A$ are defined over, the density is:

\[ \frac{1}{[F:K]}\]

\item If $A$ is isogenous to the product of a CM elliptic curve and a non-CM elliptic curve, the density is $1$ if the CM field is contained in $K$ and $1/2$ otherwise.

\item If $A$ is isogenous to the product of two CM elliptic curves with CM fields $F_1$ and $F_2$, the density is:

\[ \frac{1}{ [K F_1 F_2: K]}\]

\end{itemize}

In particular, the density is always $1$, $1/2$, or $1/4$.

\end{theo}

\begin{proof}  We can compute the density using Theorem \ref{main} in terms of the action of the $\ell$-adic monodromy group on $\wedge^2 V \otimes \chi^{-1}$  - it is the number of connected components on which the trace is not constant, divided by the number of connected components. This ratio is clearly preserved by extension of scalars from $\mathbb Q_\ell$ to $\overline{\mathbb Q}_\ell$ and from $\overline{\mathbb Q}_\ell$ to $\mathbb C$, and passage from an algebraic group $G/\mathbb C$ to the complex Lie group $G(\mathbb C)$. It is also preserved by passage from a reductive $G(\mathbb C)$ to a maximal compact subgroup $K$ over $\mathbb C$, because a maximal compact subgroup meets each component of a reductive group over $\mathbb C$ in a Zariski dense subset, the trace on a component of a maximal compact subgroup is constant if and only if the trace on the corresponding component of the complex group is constant.

The group called $G_{\ell}^{1,\Zar}$ in \citep[Definition 2.4]{fkrs} is the kernel of the similitude character from the $\ell$-adic monodromy group to $\mathbb G_m$. Because the $\ell$-adic monodromy group contains the scalars (by an argument of Deligne, \citep[2.3]{serre77}), it is equal to $G_{\ell}^{1,\Zar}$ times the group of scalars. Because the scalars act trivially on $\wedge^2 V \otimes \chi^{-1}$, we may as well work with $G_{\ell}^{1,\Zar}$. By \citep[Theorem 2.16]{fkrs}, $G_{\ell}^{1,\Zar}$ is the base change from $\mathbb Q$ to $\mathbb Q_\ell$ of a group $AST_A$, which when base changed to $\mathbb C$ has a maximal compact subgroup $ST_A$.  Because the ratio of Theorem \ref{main} is preserved by base change and by passage to maximal compact subgroups (as the $\ell$-adic monodromy group is reductive \citep[Theorem 3]{faltings86}), we may as well work with $ST_A$, the ``Sato-Tate group". 

The group $ST_A$ is classified in \citep[Theorem 4.2]{fkrs} as being one of 52 possible groups. The density of ordinary primes is the fraction of connected components with nonconstant trace on $\wedge^2 V$, for $V$ the restriction to $ST_A$ of the standard representation of $USP(4)$. This makes proving the theorem a process of checking each individual group, which can be split into cases according to the identity component of $ST_A$. The identity component is either $USP(4), SU(2)\times SU(2), SU(2), U(1) \times U(1), SU(2) \times U(1),$ or $U(1)$. A routine computation shows:

Case $USP(4), SU(2) \times SU(2), SU(2)$: In these three cases, the density is $1$. These cases occur when $A$ is a surface with endomorphism group $\mathbb Z$, a real multiplication surface, a quaternion multiplication surface, or is isogenous the product of two non-CM elliptic curves.

Case $U(1) \times U(1)$: In this case, the density is $1$ divided by the number of components. This case occurs when $A$ is isogenous to a simple CM abelian surface or a product of two non-isogenous CM curves. The number of components is equal to the degree of the field extension over which all endomorphisms are defined. In all cases of the \citep[Theorem 4.2]{fkrs} classification, this degree is either $1$, $2$, or $4$.

Note that there are some subgroups of $USP_4$ with identity component $U(1) \times U(1)$ that are never the monodromy group of an abelian surface, and these ones can have non-identity components with non-constant trace, so their density would not equal $1$ divided by the number of components. However, all the groups listed in \citep[Theorem 4.2]{fkrs} have just one component with non-constant trace.

Case $ SU(2) \times U(1)$: The density is $1$ divided by the number of components. This occurs when $A$ is isogenous to the product of a CM elliptic curve with a non-CM elliptic curve. The number of components is $1$ if the CM field is contained in $K$ and $2$ otherwise.

Case $U(1)$: This occurs when $A$ is geometrically isogenous to a product of two copies of the same CM elliptic curve. The surface is ordinary at a prime if and only if that curve is ordinary, so the density of ordinary primes is $1$ if the CM field is contained in $K$ and $1/2$ otherwise.

These facts are summarized by the statement of this theorem.

\end{proof}

\bibliographystyle{plainnat}

\bibliography{references}

\begin{thebibliography}{8}
\providecommand{\natexlab}[1]{#1}
\providecommand{\url}[1]{\texttt{#1}}
\expandafter\ifx\csname urlstyle\endcsname\relax
  \providecommand{\doi}[1]{doi: #1}\else
  \providecommand{\doi}{doi: \begingroup \urlstyle{rm}\Url}\fi

\bibitem[Bourbaki(1975)]{bourbakilie}
Nicolas Bourbaki.
\newblock \emph{{L}ie Groups and {L}ie Algebras {III}}.
\newblock Springer, 1975.

\bibitem[Deligne(1969)]{deligne69}
Pierre Deligne.
\newblock Vari\'{e}t\'{e}s ab\'{e}liennes ordinaires sur un corps fini.
\newblock \emph{Inventiones Mathematicae}, 8:\penalty0 238,243, 1969.

\bibitem[Faltings(1986)]{faltings86}
Gerd Faltings.
\newblock Finiteness theorems for abelian varieties over number fields.
\newblock In Gary Cornell and Joseph~H. Silverman, editors, \emph{Arithmetic
  Geometry}. Springer-Verlag, 1986.

\bibitem[Fit\'{e} et~al.(2012)Fit\'{e}, Kedlaya, Rotger, and Sutherland]{fkrs}
Francesc Fit\'{e}, Kiran~S. Kedlaya, Victor Rotger, and Andrew~V. Sutherland.
\newblock Sato-{T}ate distributions and {G}alois endomorphism modules in genus
  2.
\newblock \emph{Compositio Mathematica}, 148:\penalty0 1390--1442, 2012.

\bibitem[Ogus(1982)]{ogus82}
Arthur Ogus.
\newblock Hodge cycles and crystalline cohomology.
\newblock In \emph{{H}odge cycles, motives, and {S}himura varieties}, volume
  900 of \emph{Springer Lecture Notes in Mathematics}. Springer-Verlag, 1982.

\bibitem[Pink(1998)]{pink98}
Richard Pink.
\newblock {$\ell$}-adic algebraic monodromy groups, cocharacters, and the
  {M}umford-{T}ate conjecture.
\newblock \emph{J. Reine Angew. Math}, 495:\penalty0 187--237, 1998.

\bibitem[Serre(1977)]{serre77}
Jean-Pierre Serre.
\newblock Repr\'{e}sentations {$\ell$}-adiques.
\newblock In \emph{Kyoto Symposium on Algebraic Number Theory}, pages 177--193.
  Japan Society for the Promotion of Science, 1977.

\bibitem[Serre(2011)]{serre11}
Jean{-}Pierre Serre.
\newblock \emph{Lectures on ${N}_{X}(p)$}.
\newblock Research Notes in Mathematics. A K Peters/CRC Press, 2011.

\end{thebibliography}

\end{document}